\documentclass[a4paper,UKenglish,cleveref, autoref, thm-restate]{lipics-v2021}

\usepackage{enumerate,xspace}
\usepackage{amsmath,xspace,amssymb}
\usepackage[sort,nocompress]{cite} 
\nolinenumbers
\usepackage[all]{xy}

\newcommand{\D}{\ensuremath{\mathcal D}}
\newcommand{\E}{\ensuremath{\mathbb E}\xspace}

\newcommand{\X}{\ensuremath{\mathbb X}\xspace}

\newcommand\nats{\hbox{$I \kern - .38em N$}} 
\newcommand\ints{\hbox{$Z \kern - .65em Z$}} 

\title{Reverse Tangent Categories} 

\author{Geoffrey Cruttwell}{Mount Allison University, Canada \and \url{https://www.reluctantm.com/gcruttw/}}{gcruttwell@mta.ca}{https://orcid.org/0000-0001-8742-6263}{Partially funded by an NSERC Discovery Grant}

\author{Jean-Simon Pacaud Lemay\footnote{Corresponding Author}}{Macquarie University, Australia \and \url{https://sites.google.com/view/jspl-personal-webpage/}}{js.lemay@mq.edu.au}{https://orcid.org/0000-0003-4124-3722}{For this research, this author was funded by an NSERC PDF (456414649), an ARC DECRA (DE230100303), \& an AFOSR Research Grant (FA9550-24-1-0008)}

\authorrunning{G. Cruttwell and J.-S. P. Lemay} 

\Copyright{Geoffrey Cruttwell and Jean-Simon Pacaud Lemay} 



\acknowledgements{The authors would like to thank Bryce Clark for pointing out \cite{comorphisms}, which ties in nicely with the story of this paper, as well as Geoff Vooys for useful discussions, answering questions, and helping find references regarding Example \ref{ex:affine-rev}.}

\EventEditors{Aniello Murano and Alexandra Silva}
\EventNoEds{2}
\EventLongTitle{32nd EACSL Annual Conference on Computer Science Logic (CSL 2024)}
\EventShortTitle{CSL 2024}
\EventAcronym{CSL}
\EventYear{2024}
\EventDate{February 19--23, 2024}
\EventLocation{Naples, Italy}
\EventLogo{}
\SeriesVolume{288}
\ArticleNo{9}

\begin{document}

\maketitle

\begin{abstract}
Previous work has shown that reverse differential categories give an abstract setting for gradient-based learning of functions between Euclidean spaces.  However, reverse differential categories are not suited to handle gradient-based learning for functions between more general spaces such as smooth manifolds.  In this paper, we propose a setting to handle this, which we call \emph{reverse tangent categories}: tangent categories with an involution operation for their differential bundles. 
\end{abstract}

\section{Introduction}

This paper is a direct follow-up to the paper ``Reverse Differential Categories'', published in CSL 2020 \cite{reverse}, and continues a tradition of developing categorical structures to help understand and work with ideas from differential calculus in computer science, and specifically here in relation to machine learning and automatic differentiation \cite{blute2009cartesian, cockett2014differential, cat_ML, cruttwell2022monoidal, reverse, reverse_ascent, reverse_semantics}. 

Initial work on categorical formulations of differential structures in computer science focused on the so-called ``forward'' derivative. Given a map $f: A \to B$, the forward derivative is an operation that sends tangent vectors in $A$ to tangent vectors in $B$.  While there are several different (but related) categorical formulations for the forward derivative, the relevant ones for this paper are \emph{Cartesian differential categories} \cite{blute2009cartesian} and \emph{tangent categories} \cite{cockett2014differential}. 

Cartesian differential categories formalize differential calculus over Euclidean spaces. A Cartesian differential category (Ex \ref{ex:CDC}) comes equipped with a \emph{differential combinator} $\mathsf{D}$, which is an operation that for any map $f: A \to B$, produces a map $\mathsf{D}[f]: A \times A \to B$, called the derivative of $f$. Various axioms are then demanded of $\mathsf{D}$ which enforce the properties of ordinary differentiation, such as the chain rule, symmetry of mixed partial derivatives, etc. Intuitively, one considers an input $(a,v) \in A \times A$ as a point $a$ and a tangent vector $v$ to $a$, and so the derivative $\mathsf{D}[f]$ produces a tangent vector to $f(a)$ in $B$. 

However, the definition of a Cartesian differential category assumes that a tangent vector to a point is of the same type as the point. While this is true for Euclidean spaces, this is certainly not true for more general spaces such as arbitrary smooth manifolds.  To capture the operation of the forward derivative for such spaces, one can instead work with tangent categories, which formalize differential calculus over smooth manifolds. A tangent category (Def \ref{definition:tangent-category}) in particular comes equipped with an endofunctor $\mathsf{T}$, where we think of $\mathsf{T}(A)$ as the collection of all tangent vectors to points of an object $A$. As such, we may interpret $\mathsf{T}(A)$ as the abstract tangent bundle of $A$. For a map $f: A \to B$, the associated map $\mathsf{T}(f): \mathsf{T}(A) \to \mathsf{T}(A)$ is interpreted as an operation which takes a tangent vector in $A$ to a tangent vector in $B$.  This tangent bundle functor also comes equipped with various additional structure that captures ordinary properties of differentiation -- for example, functoriality of $\mathsf{T}$ corresponds to the chain rule.  Tangent categories are a direct generalization of Cartesian differential categories which allow one to work with the forward derivative in more general settings.    

However, many areas of computer science such as automatic differentiation and gradient-based learning make much more extensive use of the ``reverse'' derivative.  Given a map $f: A \to B$, the reverse derivative is an operation which sends tangent vectors in $B$ to tangent vectors in $A$.  The reverse derivative is much more computationally efficient for maps between Euclidean spaces in which the domain space is much larger than the codomain space -- which is the typical case in machine learning scenarios. The paper ``Reverse Derivative Categories'' \cite{reverse} introduced \emph{Cartesian reverse differential categories}, which provide a categorical abstraction of the reverse derivative operation over Euclidean spaces. This time, a Cartesian reverse differential category (Ex \ref{ex:CRDC}) comes equipped with a \emph{reverse differential combinator} $\mathsf{R}$ which now takes a map $f: A \to B$ and produces a map of type $\mathsf{R}[f]: A \times B \to A$, called the reverse derivative of $f$. Intuitively, one thinks of $\mathsf{R}[f]$ as taking a point of the domain and a tangent vector of the \emph{codomain} and returning a tangent vector of the domain at the point. 

However, Cartesian reverse differential categories suffer from the same problem as their forward counterpart: they assume that the tangent vectors of a point in the space are in 1-1 correspondence with points of the space itself. The objective of this paper is to do what tangent categories did for Cartesian differential categories: introduce a setting where one has a reverse derivative operation, but in such a way that tangent vectors are not assumed to be the same as points of the space itself. As such, the main contribution of this paper is the introduction of \emph{reverse tangent categories} (Sec \ref{sec:RTC}), the reverse differential counterpart of tangent categories. 

\begin{center}
\begin{tabular}{ |c|c|c| } 
 \hline
& Euclidean Spaces & Manifolds \\  \hline
 Forward derivative & Cartesian differential category & Tangent category \\ \hline 
 Reverse derivative & Cartesian reverse differential category & \textbf{Reverse tangent category} \\ 
 \hline
\end{tabular}
\end{center}

How does one go about defining a reverse tangent category? Our first attempt at defining ``reverse tangent categories'' was similar to a Cartesian reverse differential category, that is, trying to give a direct description of what this structure should look like in terms of a reverse differentiation operation. However, this has proven difficult (see Remark \ref{rem:reverse_functor}). So, instead, here we take a different approach.  Usefully, Cartesian reverse differential categories have an alternative characterization.  A Cartesian reverse differential category is precisely a Cartesian differential category equipped with a ``linear dagger'', which is an involution operation on linear maps. This linear dagger $\dagger$ is an operation which takes a map of type $f: C \times A \to B$ which is ``linear in $A$'' and transposes the linear argument to produce a map $f^{\dagger}: C \times B \to A$ which is now ``linear in $B$''. From this point of view, the reverse differential combinator is the transpose of the forward differential combinator, that is, $\mathsf{R}[f] := \mathsf{D}[f]^{\dagger}$. Using this approach as a guide, we define a reverse tangent category as a tangent category with a suitable notion of involution.  

What form should such an involution for a reverse tangent category take? One way to look at the dagger operation of a Cartesian reverse differential category is via fibrations. Indeed, the dagger operation can be viewed as an operation which goes from the canonical ``linear fibration'' \cite{reverse, cruttwell2022monoidal} of a Cartesian differential category to its dual fibration. The notion of a dual fibration \cite{jacobs1999categorical} is an operation which takes a fibration and returns another fibration in which the fibre over $A$ is the opposite category of the fibre over $A$ from the original fibration. The analog of the linear fibration in a tangent category is a fibration of differential bundles \cite{cockett2018differential} -- these abstract the notion of smooth vector bundles from ordinary differential geometry. Thus, we define a reverse tangent category (Def \ref{definition:reverse-tangent}) to consist of a tangent category equipped with an involution operation which goes from a fibration of differential bundles (Def \ref{def:systemdbun}) to its dual fibration (Prop \ref{prop:dbundual}).  

With this definition in hand, we then (i) give examples of such structure including smooth manifolds (Ex \ref{ex:sman-rev}), but also examples from algebra (Ex \ref{ex:cring-rev}) and algebraic geometry (Ex \ref{ex:affine-rev}), (ii) show precisely how this definition relates to Cartesian reverse differential categories (Ex \ref{ex:CRDC} \& Prop \ref{prop:CRTC-CRDC}), and (iii) provide some theoretical results about reverse tangent categories (Sec \ref{sec:theory}).  Thus, the purpose of this paper is to properly introduce reverse tangent categories; we do not here consider applications of reverse tangent categories to gradient-based learning on (smooth) manifolds, or the relationship of these ideas to differential programming languages (for reverse cartesian differential categories, however, see \cite{reverse_semantics}). That said, just as the original paper on Cartesian reverse differential categories inspired work on the use of such structures in gradient-based learning for Euclidean spaces \cite{reverse_ascent, cat_ML, reverse_semantics}, reverse tangent categories should provide a suitable setting to do the same for smooth manifolds (or any other ``differential'' setting).  

\textbf{Conventions:} 
In an arbitrary category, we write objects as capital letters $A,B$, etc. and maps with lower letters as $f: A \to B$. We denote identity maps as $1_A: A \to A$, and, following the conventions of previous differential/tangent category papers, we write composition diagrammatically, that is, the composite of $f: A \to B$ and $g: B \to C$ is denoted $fg: A \to C$.

\section{Forward Tangent Categories and Differential Bundles}\label{sec:tan_cats}

In this section, we review the basics of tangent categories including the definition, some key examples, and differential bundles, which will play an important role in this paper. 

\begin{definition}\label{definition:tangent-category} \cite[Def 2.3]{cockett2014differential} A \textbf{tangent structure} on a category $\mathbb{X}$ is a sextuple $\mathbb{T} := (\mathsf{T}, \mathsf{p}, \mathsf{s}, \mathsf{z}, \ell, \mathsf{c})$ consisting of: 
\begin{enumerate}[{\em (i)}]
\item An endofunctor $\mathsf{T}: \mathbb{X} \to  \mathbb{X}$, called the \textbf{tangent bundle functor};
\item A natural transformation $\mathsf{p}_A: \mathsf{T}(A) \to A$, called the \textbf{projection}, such that for each $n\in \mathbb{N}$, the $n$-fold pullback\footnote{By convention, $\mathsf{T}_0(A) = A$ and $\mathsf{T}_1(A) = \mathsf{T}(A)$} of $\mathsf{p}_A$ exists, denoted as $\mathsf{T}_n(A)$ with projections ${\rho_j: \mathsf{T}_n(A) \to \mathsf{T}(A)}$, and such that for all $m \in \mathbb{N}$, $\mathsf{T}^m$ preserves these pullbacks, that is, $\mathsf{T}^m( \mathsf{T}_n(A))$ is the $n$-fold pullback of $\mathsf{T}^m(\mathsf{p}_A)$ with projections $\mathsf{T}^m(\rho_j)$;  
\item A natural transformation\footnote{Note that by the universal property of the pullback, we can define functors $\mathsf{T}_n: \mathbb{X} \to \mathbb{X}$.} $\mathsf{s}_A: \mathsf{T}_2(A) \to \mathsf{T}(A)$, called the \textbf{sum};
\item A natural transformation $\mathsf{z}_A: A \to \mathsf{T}(A)$, called the \textbf{zero map};
\item A natural transformation $\ell_A: \mathsf{T}(A) \to \mathsf{T}^2(A)$, called the \textbf{vertical lift};
\item A natural transformation $\mathsf{c}_A: \mathsf{T}^2(A) \to \mathsf{T}^2(A)$, called the \textbf{canonical flip}; 
\end{enumerate}
such that the equalities and universal property in \cite[Def 2.3]{cockett2014differential} are satisfied. A \textbf{tangent category} is a pair $(\mathbb{X}, \mathbb{T})$ consisting of a category $\mathbb{X}$ equipped with a tangent structure $\mathbb{T}$. 
\end{definition}

Let us briefly provide some intuition for the definition of a tangent category. Tangent categories formalize the properties of the tangent bundle on smooth manifolds from classical differential geometry. As such, an object $A$ in a tangent category can be interpreted as a base space, and $\mathsf{T}(A)$ as its abstract tangent bundle. For maps, $\mathsf{T}(f)$ is interpreted as the differential of $f$, and the functoriality of $\mathsf{T}$ represents the chain rule. The projection $\mathsf{p}_A$ is the analogue of the natural projection from the tangent bundle to its base space, making $\mathsf{T}(A)$ an abstract fibre bundle over $A$. The sum $s_A$ and the zero $\mathsf{z}_A$ make $\mathsf{T}(A)$ into an additive bundle over $A$; that is, a commutative monoid in the slice category\footnote{Commutative monoids in the slice category are also called additive bundles \cite[Sec 2.1]{cockett2014differential}.} over $A$. To explain the vertical lift, recall that in differential geometry, the double tangent bundle (i.e. the tangent bundle of the tangent bundle) admits a canonical sub-bundle called the vertical bundle which is isomorphic to the tangent bundle. The vertical lift $\ell_A$ is an analogue of the embedding of the tangent bundle into the double tangent bundle via the vertical bundle. The vertical lift also satisfies a universal property, which is essential to generalize important properties of the tangent bundle from differential geometry. Lastly, the canonical flip $\mathsf{c}_A$ is an analogue of the smooth involution of the same name on the double tangent bundle, and its naturality captures the symmetry of mixed partial derivatives ($\frac{\partial f}{\partial x \partial y} = \frac{\partial f}{\partial y \partial x}$). For more details and intuition on tangent categories, see \cite[Sec 2.5]{cockett2014differential}. 

We now recall the main examples of tangent categories we will use throughout the paper. In particular, Ex \ref{example:differential-geometry} (which is arguably the canonical example of a tangent category) and Ex \ref{ex:smooth} directly relate tangent categories to differential geometry and differential calculus, Ex \ref{example:algebra} (one of the main examples in Rosický's original paper \cite[Ex 2]{rosicky1984abstract}) provides a link to commutative algebra, and Ex \ref{example:affine} relates tangent categories to algebraic geometry. For other examples of tangent categories, see \cite[Ex 2.2]{cockett2018differential}. 

\begin{example} \label{example:differential-geometry} Let $\mathsf{SMAN}$ be the category whose objects are smooth manifolds and whose maps are smooth functions. For a smooth manifold $M$ and a point $x \in M$, let $\mathsf{T}_x(M)$ be the tangent space to $M$ at $x$, and recall that the tangent bundle of $M$ is the smooth manifold $\mathsf{T}(M)$ which is the (disjoint) union of each tangent space:
\[\mathsf{T}(M) := \bigsqcup \limits_{x \in M}  \mathsf{T}_x(M)\]
So in local coordinates, elements of $\mathsf{T}(M)$ can be described as pair ${(x, \vec v) \in \mathsf{T}(M)}$ consisting of a point $x \in M$ and a tangent vector $\vec v$ at $x$. Now for a smooth function $f: M \to N$ and a point $x \in M$, there is an induced linear map ${\mathsf{T}_x(f): \mathsf{T}_x(M) \to \mathsf{T}_{f(x)}(N)}$ called the derivative of $f$ at $x$. This induces a functor $\mathsf{T}: \mathsf{SMAN} \to \mathsf{SMAN}$ which sends a smooth manifold $M$ to its tangent bundle $\mathsf{T}(M)$, and a map $f: M \to N$ to the map $\mathsf{T}(f): \mathsf{T}(M) \to \mathsf{T}(N)$, locally defined as:
\[\mathsf{T}(f)(x, \vec v) = \left( f(x), \mathsf{T}_x(f)(\vec v) \right)\] 
 To describe the rest of the tangent structure, in local coordinates, elements of $\mathsf{T}_2(M)$ are triples $(x, \vec v, \vec w) \in \mathsf{T}_2(M)$ where $x \in M$ and $\vec v, \vec w \in \mathsf{T}_x(M)$, while elements of $\mathsf{T}^2(M)$ can be represented as quadruples $(x, \vec v, \vec w, \vec u) \in \mathsf{T}^2(M)$. Thus the natural transformations are defined as follows in local coordinates: 
  \begin{gather*}
 \mathsf{p}_M(x, \vec v) = x \quad \quad \quad  \mathsf{s}_M(x, \vec v, \vec w) = (x, \vec v + \vec w) \quad \quad \quad  \mathsf{z}_{M}(x) = ( x, \vec 0) \\
 \ell_{M}(x, \vec v) = (x, \vec 0, \vec 0, \vec v) \quad \quad \quad  \mathsf{c}_{M}(x, \vec v, \vec w, \vec u) = (x, \vec w, \vec v, \vec u)  \end{gather*}
 So $(\mathsf{SMAN}, \mathbb{T})$ is a tangent category \cite[Ex 2.2.i]{cockett2018differential}.  
\end{example}

\begin{example}\label{ex:CDC} Every Cartesian differential category is a tangent category \cite[Prop 4.7]{cockett2014differential}. Where tangent categories formalize differential calculus over smooth manifolds, Cartesian differential categories instead formalize differential calculus over Euclidean spaces. Briefly, a Cartesian differential category \cite[Def 2.1.1]{blute2009cartesian} is a category $\mathbb{X}$ with finite products (where we denote the binary product by $\times$, projections $\pi_j$, and pairing operator $\langle -,- \rangle$) which comes equipped with a \textbf{differential combinator} $\mathsf{D}$ which associates to every map $f: A \to B$ a map $\mathsf{D}[f]: A \times A \to B$, called the derivative of $f$. The differential combinator $\mathsf{D}$ satisfies seven axioms which are analogues of fundamental identities of the derivative, such as the chain rule, linearity of the derivative, etc. The canonical tangent functor of a Cartesian differential category is defined as follows:
\begin{align*}
    \mathsf{T}(A) := A \times A && \mathsf{T}(f) := \left \langle \pi_0 f , \mathsf{D}[f] \right \rangle
\end{align*}
and where the rest of the tangent structure is defined in \cite[Prop 4.7]{cockett2014differential}. For more details on Cartesian differential categories, as well as examples, see \cite{blute2009cartesian, reverse}.
\end{example}

\begin{example}\label{ex:smooth} As an explicit example of a Cartesian differential category, let $\mathsf{SMOOTH}$ be the category whose objects are the Euclidean real vector spaces $\mathbb{R}^n$, and whose maps are smooth functions ${F: \mathbb{R}^n \to \mathbb{R}^m}$ between them. $\mathsf{SMOOTH}$ is a Cartesian differential category, where for a smooth function $F = \langle f_1, \hdots, f_m \rangle: \mathbb{R}^n \to \mathbb{R}^m$, its derivative ${\mathsf{D}[F]: \mathbb{R}^n \times \mathbb{R}^n \to \mathbb{R}^m}$ is defined as the $m$-tuple of sums of the partial derivatives:
\[\mathsf{D}[F](\vec x, \vec y) := \left \langle \sum \limits^n_{i=1} \frac{\partial f_1}{\partial x_i}(\vec x) y_i, \hdots, \sum \limits^n_{i=1} \frac{\partial f_n}{\partial x_i}(\vec x) y_i \right \rangle\] 
Then by the previous example, $\mathsf{SMOOTH}$ is a tangent category; its tangent structure is precisely the one for smooth manifolds in Ex \ref{example:differential-geometry}, but restricted to the Euclidean spaces $\mathbb{R}^n$. 
\end{example}

 \begin{example} \label{example:algebra} Let $k$ be a commutative ring, and $k\text{-}\mathsf{CALG}$ be the category of commutative $k$-algebras. For a commutative $k$-algebra $A$, we denote its algebra of dual numbers as:
 \[A[\epsilon] = \lbrace a + b \epsilon \vert~ a,b \in A, \epsilon^2 = 0\rbrace\] 
 Then, borrowing notation from \cite{cruttwellLemay:AlgebraicGeometry}, we have a functor $\rotatebox[origin=c]{180}{$\mathsf{T}$}: k\text{-}\mathsf{CALG} \to k\text{-}\mathsf{CALG}$ which is defined on objects as $\rotatebox[origin=c]{180}{$\mathsf{T}$}(A) := A[\epsilon]$ and for a map $f: A \to B$, ${\rotatebox[origin=c]{180}{$\mathsf{T}$}(f): A[\epsilon] \to B[\epsilon]}$ is defined as:
 \[\rotatebox[origin=c]{180}{$\mathsf{T}$}(f)(a + b\epsilon) = f(a) + f(b) \epsilon\]  
 This gives us our tangent bundle and so $(k\text{-}\mathsf{CALG}, \rotatebox[origin=c]{180}{$\mathbb{T}$})$ is a tangent category, where the remaining tangent structure is defined in \cite[Sec 3.1]{cruttwellLemay:AlgebraicGeometry}. We also note that this example can be generalized: for any symmetric monoidal category with distributive finite biproducts, its category of commutative monoids is a tangent category by generalizing the dual numbers construction. 
 \end{example}

\begin{example} \label{example:affine} For a commutative ring $k$, the category of affine schemes over $k$ is a tangent categories \cite[Sec 4]{cruttwellLemay:AlgebraicGeometry}. Famously, the category of affine schemes over $k$ is equivalent to $k\text{-}\mathsf{CALG}^{op}$, so we may describe the tangent structure in terms of commutative $k$-algebras. For a commutative $k$-algebra $A$, we denote its module of Kähler differentials as $\Omega(A)$. Then define the ``fibré tangente'' (French for tangent bundle) of $A$, or tangent algebra of $A$ \cite[Section 2.6]{jubin2014tangent}, as the free symmetric $A$-algebra over its modules of Kähler differentials 
\[ \mathsf{T}(A) := \mathsf{Sym}_A \left( \Omega(R) \right) = \bigoplus \limits_{n=0}^{\infty} \Omega(A)^{{\otimes^s_A}^n} = A \oplus \Omega(A) \oplus \left( \Omega(A) \otimes^s_A \Omega(A) \right) \oplus \hdots \]
where $\bigoplus$ is the coproduct of modules, and $\otimes^s_R$ is the symmetrized tensor product over $A$. Equivalently, $\mathsf{T}(A)$ is the $A$-algebra generated by the set $\lbrace \mathsf{d}(a) \vert~ a \in R \rbrace$ modulo the equations:
\begin{align*}
  \mathsf{d}(1) = 0 && \mathsf{d}(a+b) = \mathsf{d}(a) + \mathsf{d}(b) && \mathsf{d}(ab) = a \mathsf{d}(b) + b \mathsf{d}(a)
\end{align*}
which are the same equations that are modded out to construct the module of Kähler differentials of $A$. Thus, an arbitrary element of $\mathsf{T}(A)$ is a finite sum of monomials of the form $a \mathsf{d}(b_1) \hdots \mathsf{d}(b_n)$, and so the algebra structure of $\mathsf{T}(A)$ is essentially the same as polynomials. This also induces a functor $\mathsf{T}: k\text{-}\mathsf{CALG} \to k\text{-}\mathsf{CALG}$ which maps a commutative $k$-algebra to its tangent algebra $\mathsf{T}(A)$, and a $k$-algebra morphism $f: A \to B$ to the $k$-algebra morphism $\mathsf{T}(f): \mathsf{T}(A) \to \mathsf{T}(B)$ defined on generators as follows: 
\begin{align*}
    \mathsf{T}(f)(a) = f(a) && \mathsf{T}(f)(\mathsf{d}(a)) = \mathsf{d}(f(a))
\end{align*}
This gives us our tangent bundle and so $(k\text{-}\mathsf{CALG}^{op}, \mathbb{T})$ is a tangent category, where the remaining tangent structure is defined in \cite[Sec 4.1]{cruttwellLemay:AlgebraicGeometry}. 
 \end{example}

 There are many concepts from differential geometry which one can define in an arbitrary tangent category.  The concept that plays a crucial role in the definition of reverse tangent categories is that of a \textbf{differential bundle}, which generalizes the idea of a smooth vector bundle from differential geometry. 
 
\begin{definition} \label{def:differentialbundle} \cite[Def 2.3]{cockett2018differential} In a tangent category $(\mathbb{X}, \mathbb{T})$, a \textbf{differential bundle} is a quadruple ${\mathcal{E} = (\mathsf{q}: E \to A, \sigma: E_2 \to E, \zeta: A \to E, \lambda: E \to \mathsf{T}(E))}$ consisting of: 
\begin{enumerate}[{\em (i)}]
\item Objects $A$ and $E$ of $\mathbb{X}$;
\item A map ${\mathsf{q}: E \to A}$ of $\mathbb{X}$, called the \textbf{projection}, such that for each $n \in \mathbb{N}$, the pullback of $n$ copies of $\mathsf{q}$ exists; we denote this pullback as $E_n$ with $n$ projection maps $\pi_j: E_n \to E$, and for all $m \in \mathbb{N}$, $\mathsf{T}^m$ preserves these pullbacks;
\item A map $\sigma: E_2 \to E$ of $\mathbb{X}$, called the \textbf{sum};
\item A map $\zeta: A \to E$ of $\mathbb{X}$, called the \textbf{zero};
\item A map $\lambda: E \to \mathsf{T}(E)$ of $\mathbb{X}$, called the \textbf{lift}; 
\end{enumerate}
such that the equalities and universal property in  \cite[Def 2.3]{cockett2018differential} are satisfied. When there is no confusion, differential bundles will be written as $\mathcal{E} = (\mathsf{q}: E \to A, \sigma, \mathsf{z}, \lambda)$, and when the objects are specified simply as $\mathcal{E} = (\mathsf{q}, \sigma, \zeta, \lambda)$. If $\mathcal{E} = (\mathsf{q}: E \to A, \sigma, \mathsf{z}, \lambda)$ is a differential bundle, we also say that $\mathcal{E}$ is a \textbf{differential bundle over $A$}. 
\end{definition}

If $\mathcal{E} = (\mathsf{q}: E \to A, \sigma, \mathsf{z}, \lambda)$ is a differential bundle, the object $A$ is interpreted as a base space and the object $E$ as the total space. The projection $\mathsf{q}$ is the analogue of the bundle projection from the total space to the base space, making $E$ an ``abstract bundle over $A$''. The sum $\sigma$ and the zero $\zeta$ make each fibre into a commutative monoid. Lastly, the lift $\lambda$ is an analogue of the embedding of the total space into its tangent bundle  (sometimes called the \emph{small vertical lift}).

There are two possible kinds of morphism between differential bundles: one where the base objects can vary and one where the base object is fixed. The former is used as the maps in the category of all differential bundles in the tangent category. In either case, a differential bundle morphism is asked to preserve the projections and the lifts of the differential bundles. 

\begin{definition} \label{def:differentialbundlemorph} \cite[Def 2.3]{cockett2018differential} In a tangent category $(\mathbb{X}, \mathbb{T})$,
\begin{enumerate}[{\em (i)}]
\item A \textbf{linear differential bundle morphism} $(f,g): \mathcal{E} = (\mathsf{q}: E \to A, \sigma, \zeta, \lambda) \to \mathcal{E}^\prime = (\mathsf{q}^\prime: E^\prime \to A^\prime, \sigma^\prime, \zeta^\prime, \lambda^\prime)$ 
is a pair of maps $f: E \to E^\prime$ and $g: A \to A^\prime$ such that the following diagrams commute: 
   \begin{equation}\label{diffbunmap}\begin{gathered} 
   \xymatrixcolsep{5pc}\xymatrix{ E \ar[r]^-{f} \ar[d]_-{\mathsf{q}} & E^\prime \ar[d]^-{\mathsf{q}^\prime} &  E \ar[d]_-{\lambda} \ar[r]^-{f}  & E^\prime  \ar[d]^-{\lambda^\prime}    \\
 A \ar[r]_-{g} & A^\prime & \mathsf{T}(E) \ar[r]_-{\mathsf{T}(f)}  & \mathsf{T}(E^\prime)}  \end{gathered}\end{equation} 
 Let $\mathsf{DBun}\left[(\mathbb{X}, \mathbb{T}) \right]$ be the category whose objects are differential bundles in $(\mathbb{X}, \mathbb{T})$, and whose maps are linear differential bundle morphisms between them.
\item A \textbf{linear $A$-differential bundle morphism} $f: \mathcal{E} = (\mathsf{q}: E \to A, \sigma, \zeta, \lambda) \to \mathcal{E}^\prime = (\mathsf{q}^\prime: E^\prime \to A, \sigma^\prime, \zeta^\prime, \lambda^\prime)$ is a map ${f: E \to E^\prime}$ such that $(f,1_A): \mathcal{E} \to \mathcal{E}^\prime$ is a differential bundle morphism, that is, the following diagrams commute: 
   \begin{equation}\label{Adiffbunmap}\begin{gathered} 
   \xymatrixcolsep{5pc}\xymatrix{ E \ar[r]^-{f} \ar[dr]_-{\mathsf{q}} & E^\prime \ar[d]^-{\mathsf{q}^\prime} &  E \ar[d]_-{\lambda} \ar[r]^-{f}  & E^\prime  \ar[d]^-{\lambda^\prime}    \\
& A  & \mathsf{T}(E) \ar[r]_-{\mathsf{T}(f)}  & \mathsf{T}(E^\prime)}  \end{gathered}\end{equation} 
Let $\mathsf{DBun}\left[ A \right]$ denote the subcategory of differential bundles over $A$ and linear $A$-differential bundle morphisms between them. 
 \end{enumerate}
\end{definition}

One does not need to assume that linear differential bundle morphisms preserve the sum and zero since, surprisingly, this follows from preserving the lift \cite[Prop 2.16]{cockett2018differential}. Other properties of differential bundle morphisms can be found in \cite[Sec 2.5]{cockett2018differential}. Here are some examples of differential bundles and morphisms between them: 

\begin{example} In a tangent category $(\mathbb{X}, \mathbb{T})$, for any object $A$, its tangent bundle is a differential bundle over $A$, that is, $\mathcal{T}(A) := (\mathsf{p}_A: \mathsf{T}(A) \to A, \mathsf{s}_A, \mathsf{z}_A, \ell_A)$ is a differential bundle over $A$, and for every map $f: A \to B$, $\mathcal{T}(f) := (f, \mathsf{T}(f)): \mathcal{T}(A) \to \mathcal{T}(B)$ is a linear differential bundle morphism \cite[Ex 2.4]{cockett2018differential}. As such, we obtain a functor $\mathcal{T}: \mathbb{X} \to \mathsf{DBun}\left[(\mathbb{X}, \mathbb{T}) \right]$. 
\end{example}

\begin{example}\label{ex:smoothbundles} In $(\mathsf{SMAN}, \mathbb{T})$, differential bundles over a smooth manifold $M$ correspond precisely to smooth vector bundles over $M$.  Briefly, recall that a smooth vector bundle, in particular, consists of smooth manifolds $M$, called the base space, and $E$, called the total space, and a smooth surjection $\mathsf{q}: E \to M$, called the projection, such that for each point $x \in M$, the fibre $E_x = \lbrace e \in E \vert \mathsf{q}(e) = x \rbrace$ is a finite-dimensional $\mathbb{R}$-vector space. A smooth vector bundle morphism from $\mathsf{q}: E \to M$ to $\mathsf{q}^\prime: E^\prime \to M^\prime$ consist of two smooth functions $f: M \to M^\prime$ and $g: E \to E^\prime$ such that $f(\mathsf{q}(e)) = \mathsf{q}^\prime(g(e))$ for all $e \in E$ and the induced maps $g_x: E_x \to E^\prime_{f(x)}$ are $\mathbb{R}$-linear maps. Let $\mathsf{SVEC}$ be the category of smooth vector bundles and smooth vector bundle morphisms between them. Every smooth vector bundle over $M$ gives a differential bundle over $M$ in $(\mathsf{SMAN}, \mathbb{T})$, and vice versa. As such, there is an equivalence $\mathsf{DBun}\left[(\mathsf{SMAN}, \mathbb{T}) \right] \simeq \mathsf{SVEC}$. For full details, see \cite{macadam2021vector}.  
\end{example}

\begin{example} In general, for an arbitrary Cartesian differential category, there is not necessarily a nice full description of all differential bundles. However, there is a nice class of differential bundles which plays an important role in the story of this paper. Recall that a Cartesian differential category $\mathbb{X}$ is also a \textbf{Cartesian left additive category} \cite[Definition 1.2.1]{blute2009cartesian}, which in particular means that every homset $\mathbb{X}(A,B)$ is a commutative monoid, with binary operation $+$ and zero $0$. Then for every pair of objects $A$ and $X$, their product $(\pi_0: A \times X \to A, 1_A \times (\pi_0 + \pi_1), \langle 1_A, 0 \rangle, \langle \pi_0, 0, 0, \pi_1 \rangle)$ is a differential bundle over $A$. This generalization the notion of trivial smooth vector bundles from differential geometry. Again we stress that not all differential bundles in a Cartesian differential category are necessarily of this form. That said, for $\mathsf{SMOOTH}$, every differential bundle is of this form since smooth vector bundles over a Euclidean space is a trivial bundle. 
\end{example}

\begin{example} In $(k\text{-}\mathsf{CALG}, \rotatebox[origin=c]{180}{$\mathbb{T}$})$, differential bundles over a commutative $k$-algebra $A$ correspond precisely to $A$-modules \cite[Thm 3.10]{cruttwellLemay:AlgebraicGeometry}. Briefly, let $\mathsf{MOD}$ be the category whose objects are pairs $(A, M)$ consisting of a commutative $k$-algebra $A$ and an $A$-module $M$, and whose maps $(f,g): (A,M) \to (B,N)$ consist of a $k$-algebra morphism $f: A \to B$ and a $k$-linear map $g: M \to N$ such that $g(am) = f(a)g(m)$ for all $a \in A$ and $m \in M$. Then there is an equivalence $\mathsf{DBun}\left[(k\text{-}\mathsf{CALG}, \rotatebox[origin=c]{180}{$\mathbb{T}$}) \right] \simeq \mathsf{MOD}$, which in particular sends an $A$-module $M$ to its nilpotent extension:
\[M[\varepsilon] = \lbrace a + m \varepsilon \vert~ a \in A, m \in M, \varepsilon^2 =0  \rbrace\]
For full details, see \cite[Sec 3]{cruttwellLemay:AlgebraicGeometry}. 
\end{example}

\begin{example}\label{ex:affine-mod} In $(k\text{-}\mathsf{CALG}^{op}, \mathbb{T}^{op})$, differential bundles over a commutative $k$-algebra $A$ again correspond precisely to $A$-modules \cite[Thm 4.15]{cruttwellLemay:AlgebraicGeometry}. However this time, there is an equivalence $\mathsf{DBun}\left[(k\text{-}\mathsf{CALG}^{op}, \mathbb{T}^{op}) \right] \simeq \mathsf{MOD}^{op}$ (or in other words $\mathsf{DBun}\left[(k\text{-}\mathsf{CALG}^{op}, \mathbb{T}^{op}) \right]^{op} \simeq \mathsf{MOD}$), which in particular sends an $A$-module $M$ to the free symmetric $A$-algebra over $M$: 
\[ \mathsf{Sym}_A \left( M \right) = \bigoplus \limits_{n=0}^{\infty} M^{{\otimes^s_A}^n} = A \oplus M \oplus \left( M \otimes^s_A M \right) \oplus \hdots \]
For full details, see \cite[Sec 4]{cruttwellLemay:AlgebraicGeometry}. 
\end{example}

We conclude this section with some results about differential bundles which we will need in the later sections. The first is that the tangent bundle of a differential bundle is also a differential bundle, and the second is that the pullback along the projection of a differential bundle is again a differential bundle. 

\begin{proposition}\label{prop:dbun-tan}\cite[Lem 2.5]{cockett2018differential} In a tangent category $(\mathbb{X}, \mathbb{T})$, if $\mathcal{E} = (\mathsf{q}: E \to A, \sigma, \zeta, \lambda)$ is a differential bundle, then $\overline{\mathsf{T}}(\mathcal{E}) := \left(\mathsf{T}(\mathsf{q}): \mathsf{T}(E) \to \mathsf{T}(A), \mathsf{T}(\sigma), \mathsf{T}(\zeta), \mathsf{T}(\lambda)\mathsf{c}_E  \right)$ is a differential bundle\footnote{It is important to note that the canonical flip is used to the define the lift for the tangent bundle of a differential bundle.}, which we call the tangent bundle of $\mathcal{E}$. Similarly, if $(f,g): \mathcal{E} \to \mathcal{E}^\prime$ is a linear differential bundle morphism, then $\overline{\mathsf{T}}(f,g) := \left( \mathsf{T}(f), \mathsf{T}(g) \right): \overline{\mathsf{T}}(\mathcal{E}) \to \overline{\mathsf{T}}(\mathcal{E}^\prime)$ is a linear differential bundle morphism. This induces a functor $\overline{\mathsf{T}}: \mathsf{DBun}\left[(\mathbb{X}, \mathbb{T}) \right] \to \mathsf{DBun}\left[(\mathbb{X}, \mathbb{T}) \right]$. 
\end{proposition}

\begin{proposition}\label{prop:dbun-pb}\cite[Lem 2.7]{cockett2018differential} In a tangent category $(\mathbb{X}, \mathbb{T})$, let $\mathcal{E} = (\mathsf{q}: E \to A, \sigma, \zeta, \lambda)$ be a differential bundle and $f: X \to A$ is a map such that for each $n \in \mathbb{N}$, the pullback of $n$ copies of $\mathsf{q}$ along $f$ exists, which we denote as $X \times_M E_n$ with projection maps $\pi_0: X \times_M E_n \to X$ and $\pi_{n+1}: X \times_M E_n \to E$, and for all $m \in \mathbb{N}$, $\mathsf{T}^m$ preserves these pullbacks. Then $X \times_M \mathcal{E} := \left( \pi_0: X \times_M E_n \to X, 1_X \times_M \sigma, 1_X \times_M \zeta, \mathsf{0}_X \times_M \lambda \right) $ is a differential bundle and $(\pi_0, \pi_1): X \times_M \mathcal{E} \to \mathcal{E}$ is a linear differential bundle morphism. 
\end{proposition}

\section{Reverse Tangent Categories}\label{sec:RTC}

In this section, we introduce the main novel concept of this paper: \emph{reverse} tangent categories. As explained in the introduction, the way we define reverse tangent categories is by generalizing the definition of a Cartesian reverse differential category as a Cartesian differential category with an involution from its linear fibration to its dual fibration (which we review in Ex \ref{ex:CRDC}). For a tangent category, the analogue of the linear fibration is replaced by a suitable fibration of differential bundles. As such, a reverse tangent category is a tangent category with an involution from its differential bundle fibration to its dual fibration. For a review of fibrations and their basic theory, see \cite{jacobs1999categorical}.

We first wish to build a fibration of differential bundles. Unfortunately, for an arbitrary tangent category, the forgetful functor from its category of differential bundles is not necessarily a fibration (as tangent categories do not assume that the necessary pullbacks exist). As such, we will need to specify a class of differential bundles that do form a fibration -- which we call a system of differential bundles.

\begin{definition}\label{def:systemdbun} For a tangent category $(\mathbb{X}, \mathbb{T})$, a \textbf{system of differential bundles} consists of a collection of differential bundles, $\D$, such that:
\begin{enumerate}[{\em (i)}]
\item For every object $A$, $\mathcal{T}(A) \in \mathcal{D}$;
\item If $\mathcal{E} \in \mathcal{D}$, then so is its tangent bundle $\overline{\mathsf{T}}(\mathcal{E}) \in \mathcal{D}$ (see Prop. \ref{prop:dbun-tan});
\item If $\mathcal{E} \in \mathcal{D}$, then for any map $f: X \to A$ in $\mathbb{X}$, the pullback of $\mathcal{E}$ along $f$ exists, and $X \times_A \mathcal{E} \in \mathcal{D}$ (see Prop. \ref{prop:dbun-pb}).
\end{enumerate}
Let $\mathsf{DBun}_\mathcal{D}\left[(\mathbb{X}, \mathbb{T}) \right]$ be the full subcategory of $\mathsf{DBun}\left[(\mathbb{X}, \mathbb{T}) \right]$ whose objects are the differential bundles in $\mathcal{D}$, and let $\mathsf{DBun}_\mathcal{D}\left[A\right]$ be the full subcategory of $\mathsf{DBun}\left[A\right]$ whose objects are the differential bundles over $A$ in $\mathcal{D}$. We denote the forgetful functor as $\mathsf{U}_\mathcal{D}: \mathsf{DBun}_\mathcal{D}\left[(\mathbb{X}, \mathbb{T}) \right] \to \mathbb{X}$ which maps a differential bundles to its base object. 
\end{definition}

\begin{proposition}
If $\D$ is a system of differential bundles on a tangent category $(\mathbb{X}, \mathbb{T})$, then the forgetful functor $\mathsf{U}_\mathcal{D}: \mathsf{DBun}_\mathcal{D}\left[(\mathbb{X}, \mathbb{T}) \right] \to \mathbb{X}$ is a fibration, and the fibre over an object $A$ is $\mathsf{DBun}_\mathcal{D}\left[A\right]$. 
\end{proposition}
\begin{proof}
This is immediate since $\D$ is closed under pullbacks; the Cartesian morphisms are precisely the pullback squares of differential bundles. 
\end{proof}

\begin{example} For the tangent categories $(\mathsf{SMAN}, \mathbb{T})$, $(k\text{-}\mathsf{CALG}, \rotatebox[origin=c]{180}{$\mathbb{T}$})$, and $(k\text{-}\mathsf{CALG}^{op}, \mathbb{T}^{op})$,  the class of all differential bundles form a system of differential bundles. These recreate the previously known results that $\mathsf{SVEC}$ is a fibration over $\mathsf{SMAN}$, that $\mathsf{MOD}$ is a fibration over $k\text{-}\mathsf{CALG}$, and that $\mathsf{MOD}^{op}$ is a fibration over $k\text{-}\mathsf{CALG}^{op}$ (or in other words, that $\mathsf{MOD}$ is a cofibration over $k\text{-}\mathsf{CALG}$). 
\end{example}

\begin{example}\label{ex:P} For a Cartesian differential category $\mathbb{X}$, the differential bundles of the form $\pi_0: A \times X \to A$ form a system of differential bundles for $(\mathbb{X}, \mathbb{T}_\mathsf{D})$. The resulting fibration for $\mathcal{P}$ corresponds to the canonical \textbf{linear fibration} \cite[Def 16]{cruttwell2022monoidal} of a Cartesian differential category, which plays a key role in characterizing Cartesian reverse differential categories. The linear fibration $\mathcal{L}\left[ \mathbb{X} \right]$ has objects pairs $(C,A)$ of objects of $\mathbb{X}$ and whose maps are pairs $(f,g): (C,A) \to (D,B)$ where $f: C \to D$ and $g: C \times A \to B$ is \textbf{linear in its second argument} \cite[Def 15]{cruttwell2022monoidal}, that is, $\langle \pi_0, 0, 0, \pi_1 \rangle \mathsf{D}[g] = g$. Relating this back to $\mathcal{P}$, it is straightforward to work out that by the two diagrams of (\ref{diffbunmap}), a linear differential bundle morphism of type $\left( \pi_0: A \times X \to A \right) \to \left( \pi_0: A^\prime \times X^\prime \to A^\prime \right)$ corresponds to a map $A \to A^\prime$ and a map $A \times X \to X^\prime$ which is linear in its second argument. Thus $\mathsf{DBun}_\mathcal{P}\left[(\mathbb{X}, \mathbb{T}_\mathsf{D}) \right] \cong \mathcal{L}\left[ \mathbb{X} \right]$ are equivalent as fibrations. 
\end{example}

\begin{example}\label{ex:P-SMOOTH} As an example of the above, in $\mathsf{SMOOTH}$, the notion of linearity in the Cartesian differential category sense corresponds precisely to the classical notion of $\mathbb{R}$-linearity. Explicitly, a smooth function $F: \mathbb{R}^n \times \mathsf{R}^m \to \mathbb{R}^k$ is linear in its second argument in the above sense if and only if $F$ is $\mathbb{R}$-linear in its second argument, that is, $F(\vec x, r \vec y + s \vec z) = r F(\vec x, \vec y) + s F(\vec x, \vec z)$. Now recall that every smooth vector bundle over $\mathbb{R}^n$ is a trivial vector bundle, and thus (up to isomorphism) of the form $\mathbb{R}^n \times \mathbb{R}^m$. Therefore, $\mathcal{P}$ is precisely the class of all differential bundles of $(\mathsf{SMOOTH}, \mathbb{T}_\mathsf{D})$, and so we have that $\mathsf{DBun}_\mathcal{P}\left[(\mathsf{SMOOTH}, \mathbb{T}_\mathsf{D}) \right] = \mathsf{DBun}\left[(\mathsf{SMOOTH}, \mathbb{T}_\mathsf{D}) \right] \cong \mathcal{L}\left[ \mathbb{X} \right]$.   
\end{example}

We now turn our attention to the notion of the dual fibration. For a fibration $\mathsf{F}: \mathbb{E} \to \mathbb{X}$, its \textbf{dual fibration} \cite[Def 1.10.11]{jacobs1999categorical} is the fibration $\mathsf{F}^\circ: \mathbb{E}^\circ \to \mathbb{X}$ where:
\begin{enumerate}[{\em (i)}]
    \item The objects of $\mathbb{E}^\circ$ are the same as the objects of $\mathbb{E}$;  
    \item A map from $A$ to $B$ in $\E^\circ$ consists of an equivalence class of pairs $[(v,c)]$ where $v: C \to A$ is vertical and $c: C \to B$ is Cartesian, and where two pairs $(v,c)$ and $(v^\prime, c^\prime)$ are equivalent if there is a vertical isomorphism which make the relevant triangles commute;
    \item $\mathsf{F}^\circ: \mathbb{E}^\circ \to \mathbb{X}$ is defined on objects as $\mathsf{F}$ and on maps as $\mathsf{F}^\circ\left( [(v,c)] \right) = \mathsf{F}(v)$. 
\end{enumerate}
The dual fibration of $F$ has the key property that its fibre over $A$ is the \emph{opposite} category of the fibre of $F$ over $A$.  For more details about the dual of a fibration, see \cite[Sec 1.10.11 -- 13]{jacobs1999categorical}. For a system of differential bundles, we can explicitly work out what its dual fibration will be. As this fibration consists of objects and maps as in the arrow category, this is essentially a modified version of the dual fibration of the arrow category; this dual fibration is known in various places as the category of containers \cite{containers} or dependent lenses \cite{gen_lenses}, and has a close relationship to polynomial functors.  In our case, the dual fibration of a system of differential bundles has the following form: 

\begin{proposition}\label{prop:dbundual} If $\D$ is a system of differential bundles on a tangent category $(\mathbb{X}, \mathbb{T})$, then the dual fibration $\mathsf{U}^\circ_\mathcal{D}: \mathsf{DBun}^\circ_\mathcal{D}\left[(\mathbb{X}, \mathbb{T}) \right] \to \mathbb{X}$ consists of:
\begin{enumerate}[{\em (i)}]
    \item Objects are differential bundles $\mathcal{E} \in \mathcal{D}$; 
    \item A map $(f,g): \mathcal{E} = (\mathsf{q}: E \to A, \sigma, \zeta, \lambda) \to \mathcal{E}^\prime = (\mathsf{q}^\prime: E^\prime \to A^\prime, \sigma^\prime, \zeta^\prime, \lambda^\prime)$ in $\mathsf{DBun}^\circ_\mathcal{D}\left[(\mathbb{X}, \mathbb{T}) \right]$ is a pair consisting of a map $f: A \to A^\prime$ of $\mathbb{X}$ and a linear $A^\prime$-differential bundle morphism $g: A \times_{A^\prime} \mathcal{E}^\prime \to \mathcal{E}$, where $A \times_{A^\prime} \mathcal{E}^\prime$ is the pullback bundle of $\mathcal{E}^\prime$ along $f$; that is, a map  $g: A \times_{A^\prime} E^\prime \to E$ in $\mathbb{X}$ such that the following diagrams commmute:     
    \begin{equation}\label{Adiffbunmap-dual}\begin{gathered} 
   \xymatrixcolsep{5pc}\xymatrix{ A \times_{A^\prime} E^\prime \ar[r]^-{g} \ar[dr]_-{\pi_0} & E \ar[d]^-{\mathsf{q}} &  A \times_{A^\prime} E^\prime \ar[d]_-{\mathsf{z} \times_{A^\prime} \lambda} \ar[r]^-{g}   & E \ar[d]^-{\lambda}   \\
& A  & \mathsf{T}\left( A \times_{A^\prime} E^\prime \right)\ar[r]_-{\mathsf{T}(g)}  &  \mathsf{T}(E)  }  \end{gathered}\end{equation}
\item The identity on $\mathcal{E}$ is the pair $(1_A: A \to A, \pi_1: A \times_A E \to E)$; 
\item The composition of maps $(f: A \to A^\prime, g: A \times_{A^\prime} E^\prime \to E): \mathcal{E} \to \mathcal{E}^\prime$ and \\
${(h: A^\prime \to A^{\prime \prime},{k: A^\prime \times_{A^{\prime \prime}} E^{\prime \prime} \to E^\prime}): \mathcal{E}^\prime  \to \mathcal{E}^{\prime \prime}}$ is the pair
    \begin{equation}\label{dual-comp}\begin{gathered} 
 \left( fh,  \xymatrixcolsep{6pc}\xymatrix{ A \times_{A^{\prime \prime}} E^{\prime \prime} \ar[r]^-{\left \langle 1_A, (f \times_{A^\prime} 1_{E^{\prime \prime}}) k \right \rangle}& A \times_{A^\prime} E^\prime  \ar[r]^-{g}  & E  } \right):  \mathcal{E} \to \mathcal{E}^{\prime \prime}  \end{gathered}\end{equation}
\item $\mathsf{U}^\circ_\mathcal{D}: \mathsf{DBun}_\mathcal{D}\left[(\mathbb{X}, \mathbb{T}) \right] \to \mathbb{X}$ is defined on objects $ \mathcal{E} = (\mathsf{q}: E \to A, \sigma, \zeta, \lambda)$ as ${\mathsf{U}^\circ_\mathcal{D}(\mathcal{E}) = A}$, and on maps as ${\mathsf{U}^\circ_\mathcal{D}(f,g) = f}$. 
\end{enumerate}
Furthermore, note that when $A = A^\prime$, for maps in $\mathsf{DBun}^\circ_\mathcal{D}\left[(\mathbb{X}, \mathbb{T}) \right]$ of the form $(1_A, g): \mathcal{E} \to \mathcal{E}^\prime$, the domain of the second component is simply $E^\prime$, so we have that $g: E^\prime \to E$. As such, the fibre over an object $A$ is simply the opposite category $\mathsf{DBun}^{op}_\mathcal{D}\left[A\right]$. 
\end{proposition}

It is worth emphasizing the form composition takes in the dual fibration: it is a mixture of a ``forward'' composite in the first component, and a ``reverse'' composite in the second component (with the $k$ appearing before the $g$).  

\begin{example} In \cite{comorphisms}, Higgins and Mackenzie call maps in the dual fibration of $\mathsf{SVEC}$ ``comorphisms'' of vector bundles \cite[Def 1.1]{comorphisms}, and similarly call maps in the dual fibration of $\mathsf{MOD}^{op}$ ``comorphisms'' of modules \cite[Def 2.1]{comorphisms} (though the categories of ``comorphisms'' they define are the opposites of the ones we define in this paper). 
\end{example}

The last ingredient in the definition of a reverse tangent category is an \emph{involution} on the differential bundle fibration. Unlike for a Cartesian reverse differential category where the involution is a ``dagger'' and is asked to be the identity on objects \cite[Def 33]{reverse}, the involution for a reverse tangent category need not be the identity on objects but is still required to be reflexive, which is captured by a natural transformation, and be compatible with the tangent bundle functor. Since a system of differential bundles $\D$ is closed under the tangent bundle, the induced functor $\overline{\mathsf{T}}$ from Prop \ref{prop:dbun-tan} restricts to the class of $\D$, so we have a functor $\overline{\mathsf{T}}_\D: \mathsf{DBun}_\D\left[(\mathbb{X}, \mathbb{T}) \right] \to \mathsf{DBun}_\D\left[(\mathbb{X}, \mathbb{T}) \right]$, which is clearly also a fibration morphism. It follows from \cite[Lem 32]{cruttwell2022monoidal} that we also have a tangent bundle functor on the dual fibration, $\overline{\mathsf{T}}^\circ_\D: \mathsf{DBun}^\circ_\D\left[(\mathbb{X}, \mathbb{T}) \right] \to \mathsf{DBun}^\circ_\D\left[(\mathbb{X}, \mathbb{T}) \right]$.  

\begin{definition} \label{definition:involution} For a tangent category $(\mathbb{X}, \mathbb{T})$ with a system of differential bundles $\mathcal{D}$, a \textbf{linear involution} is a pair $(\ast, \iota)$ consisting of a fibration morphism $(-)^\ast:  \mathsf{DBun}_\mathcal{D}\left[(\mathbb{X}, \mathbb{T}) \right] \to  \mathsf{DBun}^\circ_\mathcal{D}\left[(\mathbb{X}, \mathbb{T}) \right]$ where: 
\begin{enumerate}[{\em (i)}]
    \item The image of a differential bundle $\mathcal{E} = (\mathsf{q}: E \to A, \sigma, \zeta, \lambda)$ is the differential bundle denoted as $\mathcal{E}^\ast = (\mathsf{q}^\ast: E^\ast \to A, \sigma^\ast, \zeta^\ast, \lambda^\ast)$, and $\mathcal{E}^\prime$ is called the \textbf{dual bundle} of $\mathcal{E}$;
    \item The image of a linear differential bundle morphism $(f,g):  \mathcal{E} = (\mathsf{q}: E \to A, \sigma, \zeta, \lambda) \to \mathcal{E}^\prime = (\mathsf{q}^\prime: E^\prime \to A^\prime, \sigma^\prime, \zeta^\prime, \lambda^\prime)$ is denoted as $(f,g)^\ast= (f,g^\ast)$ where $g^\ast: A \times_{A^\prime} {E^\prime}^\ast \to E^\ast$.
\end{enumerate}
and a natural isomorphism $\iota_\mathcal{E}: \mathcal{E} \to \mathcal{E}^{\ast \ast}$ such that the following diagram should commute:
    \begin{equation}\label{T*-commute}\begin{gathered}  \xymatrixcolsep{5pc}\xymatrix{ \mathsf{DBun}_\mathcal{D}\left[(\mathbb{X}, \mathbb{T}) \right] \ar[r]^{(-)^*} \ar[d]_{\overline{\mathsf{T}}} & \mathsf{DBun}^\circ _\mathcal{D}\left[(\mathbb{X}, \mathbb{T}) \right]\ar[d]^{\overline{\mathsf{T}}^\circ} \\ \mathsf{DBun}_\mathcal{D}\left[(\mathbb{X}, \mathbb{T}) \right]  \ar[r]_{(-)^*} & \mathsf{DBun}_\mathcal{D}\left[(\mathbb{X}, \mathbb{T}) \right]  } 
    \end{gathered}\end{equation} 
\end{definition}

\begin{definition} \label{definition:reverse-tangent} A \textbf{reverse tangent category} is a quintuple $(\mathbb{X}, \mathbb{T}, \mathcal{D}, \ast, \iota)$ consisting of a tangent category $(\mathbb{X}, \mathbb{T})$ with a system of differential bundles $\mathcal{D}$ and a linear involution $(\ast, \iota)$.
\end{definition}

On the fibers, the involution induces an involutive functor $(-)^\ast: \mathsf{DBun}^{op}_\mathcal{D}\left[A\right] \to \mathsf{DBun}_\mathcal{D}\left[A\right]$. So for an $A$-linear differential bundle morphism $g: \mathcal{E} \to \mathcal{E}^\prime$, applying the involution will result in an $A$-linear bundle morphism of type $g^\ast: {\mathcal{E}^\prime}^\ast \to \mathcal{E}^\ast$, so in particular its underlying map is $g^\ast: {E^\prime}^\ast \to E^\ast$. 

Essentially, a reverse tangent category is a tangent category where for every differential bundle in the specified system there is a chosen ``dual'' differential bundle. In particular, we can consider the dual of the tangent bundle of an object, which we call the \emph{reverse} tangent bundle of that object.  

\begin{definition} In a reverse tangent category $(\mathbb{X}, \mathbb{T}, \mathcal{D}, \ast, \iota)$, define the \textbf{reverse tangent functor} $\mathcal{T}^\ast: \mathbb{X} \to \mathsf{DBun}_\mathcal{D}\left[(\mathbb{X}, \mathbb{T}) \right]^\circ$ as the following composite: 
    \begin{equation}\label{revtanfun}\begin{gathered} \mathcal{T}^\ast :=  \xymatrixcolsep{5pc}\xymatrix{ \mathbb{X} \ar[r]^-{\mathcal{T}} & \mathsf{DBun}_\mathcal{D}\left[(\mathbb{X}, \mathbb{T}) \right]   \ar[r]^-{(-)^*} & \mathsf{DBun}_\mathcal{D}\left[(\mathbb{X}, \mathbb{T}) \right]^\circ  } 
    \end{gathered}\end{equation}
    where on objects we denote:
    \[ \mathcal{T}^\ast(A) := \left(\mathsf{p}^\ast_A: \mathsf{T}^\ast(A) \to A, \mathsf{s}^\ast_A: \mathsf{T}_2^\ast(A) \to \mathsf{T}^\ast(A), \mathsf{z}^\ast_A:  A \to \mathsf{T}^\ast(A), \ell^\ast_A: \mathsf{T}^\ast(A) \to \mathsf{T}\mathsf{T}^\ast(A) \right)\] 
 and on maps we denote 
     \[ \mathcal{T}^\ast(f) = \left (f: A \to B,  \mathsf{T}^\ast(f): A \times_B \mathsf{T}^\ast(B) \to \mathsf{T}^\ast(A) \right)\] 
The object $\mathsf{T}^\ast(A)$ is called the \textbf{reverse tangent bundle} of $A$. 
\end{definition}

It is worth pointing out what the reverse tangent bundle functor does on maps. Given a map $f: A \to B$, we have that the second component of its image is a map of type $\mathsf{T}^\ast(f): A \times_B \mathsf{T}^\ast(B) \to \mathsf{T}^\ast(A)$. This is exactly what one wants for gradient-based learning: given a point $x$ of the domain and a cotangent vector over $f(x)$ in the \emph{codomain}, $\mathsf{T}^\ast(f)$ produces a cotangent vector in the \emph{domain}. 

\begin{remark}\label{rem:reverse_functor} It is also important to note that in general, the reverse tangent bundle does not induce a functor on the base category, since $\mathsf{T}^\ast$ is not functorial (either in the covariant or contravariant sense) with respect to the composition in $\mathbb{X}$. Moreover, the fact that $\mathsf{T}^\ast$ does not induce a functor on the base category is part of the reason why we have found it difficult to provide a direct description of a reverse tangent category (as was done for reverse cartesian differential categories in \cite{reverse}).  A tangent category has as part of its data natural transformations related to iterates of the tangent bundle functor.  Appropriate analogs of these for reverse tangent categories have thus been difficult to find given that there is no endofunctor to iterate in the reverse situation.  This is why we have instead chosen to define a reverse tangent as a tangent category with an appropriate involution operation.    
\end{remark}

We conclude this section with our main examples of reverse tangent categories. 

\begin{example} \label{ex:sman-rev} \normalfont Smooth manifolds form a reverse tangent category. Let us first describe the dual fibration of smooth vector bundles. Observe that if $\mathsf{q}^\prime: F \to N$ is a smooth vector bundle and $f: M \to N$ is a smooth function, in local coordinates, elements of the pullback $M \times_N F$ are pairs $(x,v)$ where $x \in M$ and $v \in F_{f(x)}$. As such, the dual fibration $\mathsf{SVEC}^\circ$, also called the category of star bundles \cite[Sec 41.1]{natural}, has objects smooth vector bundles $\mathsf{q}: E \to M$, and a map $\mathsf{q}: E \to M$ to $\mathsf{q}^\prime: F \to N$ consists of smooth functions $f: M \to N$ and $g: M \times_N F \to E$ such that for every $x \in M$, $\mathsf{q}(g(x,v)) = x$, and the induced map $g_x: F_{f(x)} \to E_x$ is $\mathbb{R}$-linear, so we may write $g(x,v) = (x, g_x(v))$. There is an involution which sends a smooth vector bundle $\mathsf{q}: E \to M$ to the classical dual bundle $\mathsf{q}^\ast: E^\ast \to M$ from differential geometry, where the fibres of $E^\ast$ are the $\mathbb{R}$-linear duals of the fibre of $E$, that is, $E^\ast_x$ is the dual vector space of $E_x$ in the classical linear algebra sense: 
\[E^\ast_x = \lbrace \phi: E_x \to \mathbb{R} \vert~ \text{$\phi$ $\mathbb{R}$-linear} \rbrace\] 
The involution sends a smooth vector bundle morphism $(f: M\to N, g: E \to F)$ to the pair $(f, g^\ast: M \times_{N} F^\ast \to E^\ast )$, where in local coordinates:
\[g^\ast(x, \phi) = (x, \phi( g_x( -) )\]
So $(\mathsf{SMAN}, \mathbb{T}, \mathcal{D}, \ast, \iota)$ is a reverse tangent category. The reverse tangent bundle is given by the classical \emph{cotangent bundle} $\mathsf{T}^\ast(M)$ from differential geometry, where: 
\[\mathsf{T}^\ast(M) := \bigsqcup \limits_{x \in M}  \mathsf{T}^\ast_x(M)\]
So in local coordinates, elements of the cotangent bundle $\mathsf{T}^\ast(M)$ are pairs $(m, \phi)$ where $m \in M$ and $\phi \in \mathsf{T}^\ast_m(M)$, so $\phi: \mathsf{T}_m(M) \to \mathbb{R}$ is an $\mathbb{R}$-linear morphism. For a smooth function $f: M \to N$, we have that its image via the reverse tangent bundle is of type $\mathsf{T}^\ast(f): M \times_N \mathsf{T}^\ast(N) \to \mathsf{T}^\ast(M)$. In local coordinates, elements of $ M \times_N \mathsf{T}^\ast(N)$ are pairs $(m, \phi: \mathsf{T}_{f(m)}(N) \to \mathbb{R})$, and so $\mathsf{T}^\ast(f)$ is defined is given as follows: 
\[\mathsf{T}^\ast(f)(m, \phi) = \left( m, \phi\left( T_m(f) (-) \right) \right)\] 
\end{example}

\begin{example}\label{ex:CRDC} Every Cartesian reverse differential category is a reverse tangent category. Briefly, a Cartesian reverse differential category \cite[Def 13]{reverse}  can be defined as a category with finite products which in particular comes equipped with a \textbf{reverse differential combinator} $\mathsf{R}$ which associates every map $f: A \to B$ to a map of type $\mathsf{R}[f]: A \times B \to A$, called the reverse derivative of $f$. Alternatively, a Cartesian reverse differential category can be equivalently characterized as a Cartesian differential category equipped with a linear dagger \cite[Thm 42]{reverse}. Briefly, for a Cartesian differential category $\mathbb{X}$, a \textbf{linear dagger} \cite[Def 39]{reverse} is a fibration morphism $(-)^\dagger: \mathcal{L}[\mathbb{X}]^\circ \to \mathcal{L}[\mathbb{X}]$ which is the identity on objects and involutive, such that each fibre of the linear fibration has dagger biproducts -- the dual of the linear fibration is described in detail in \cite[Ex 34]{cruttwell2022monoidal}. In particular, the linear dagger allows one to transpose linear arguments, that is, for every map $g: C\times A \to B$ which is linear in its second argument, the linear dagger gives a map $g^\dagger: C \times B \to A$ which is linear in its second argument. Now by using the linear dagger $\dagger$ for the linear involution and setting $\iota =1$ (since the dagger is the identity on objects), we thus have that $(\mathbb{X}, \mathbb{T}_\mathsf{D}, \mathcal{P}, \dagger, 1)$ is a reverse tangent category. In particular, the reverse tangent bundle of an object $A$ is $\mathsf{T}^\ast(A) = A \times A$, while for a map $f: A \to B$, its image via the reverse tangent bundle is of type $\mathsf{T}^\ast(f): A \times B \to A \times A$ and can be nicely expressed using the reverse differential combinator as follows:
\[\mathsf{T}^\ast(f) = \langle \pi_0, \mathsf{R}[f] \rangle\] 
\end{example}

\begin{example} $\mathsf{SMOOTH}$ is a Cartesian reverse differential category where for a smooth function $F: \mathbb{R}^n \to \mathbb{R}^m$, which recall is an $m$-tuple $F = \langle f_1, \hdots, f_m \rangle$ of smooth functions ${f_i: \mathbb{R}^n \to \mathbb{R}}$, its reverse derivative ${\mathsf{R}[F]: \mathbb{R}^n \times \mathbb{R}^m \to \mathbb{R}^n}$ is defined as the $n$-tuple:
\[\mathsf{R}[F](\vec x, \vec z) := \left \langle \sum \limits^m_{i=1} \frac{\partial f_i}{\partial x_1}(\vec x) z_i, \hdots, \sum \limits^m_{i=1} \frac{\partial f_i}{\partial x_n}(\vec x) z_i \right \rangle\] 
Thus $(\mathsf{SMOOTH}, \mathbb{T}_\mathsf{D}, \mathcal{P}, \dagger, 1)$ is a reverse tangent category, where in particular:
\begin{align*}
    \mathsf{T}^\ast(\mathbb{R}^n) = \mathbb{R}^n \times \mathbb{R}^n && \mathsf{T}^\ast(F)(\vec x, \vec z) = \left( \vec x, \mathsf{R}[F](\vec x, \vec z) \right)
\end{align*}
\end{example}

\begin{example} \label{ex:cring-rev} \normalfont All of $k\text{-}\mathsf{CALG}$ will not form a reverse tangent category; instead, we must restrict to those algebras which are finitely generated free $k$-modules. So, let $k\text{-}\mathsf{CALG}_{f}$ and $\mathsf{MOD}_f$ be the full subcategories whose objects are finitely generated free $k$-modules. Clearly we still have that $(k\text{-}\mathsf{CALG}_f, \rotatebox[origin=c]{180}{$\mathbb{T}$})$ is a tangent category and $\mathsf{DBun}\left[(k\text{-}\mathsf{CALG}_f, \rotatebox[origin=c]{180}{$\mathbb{T}$}) \right] \simeq \mathsf{MOD}_f$ is still a fibration of differential bundles, and thus all differential bundles form a system of differential bundles. The dual fibration $\mathsf{MOD}_f^\circ$ has objects pairs $(A,M)$ consisting of a commutative $k$-algebra $A$ and an $A$-module $M$, and where a map is now a pair $(f,g): (A,M) \to (B,N)$ consisting of a $k$-algebra morphism $f: A \to B$ and a $k$-linear map $g: N \to M$ such that $g(f(a) n) = a g(n)$. Since an $A$-module $M$ is also a $k$-module, it make sense to consider the $k$-linear dual of $M$, which we denote as:
\[M^\circledast = \lbrace \phi: M \to k \vert~ \text{$\phi$ $k$-linear} \rbrace\] 
Note that $M^\circledast$ is also an $A$-module via the action $(a, \phi) \mapsto \phi(a \cdot -)$. Since $M$ is a finitely generated free $k$-module, it is reflexive so we have the canonical isomorphism $M \cong M^{\circledast\circledast}$. As such, this induces an involution which sends objects $(A,M)$ to $(A, M^\circledast)$, and maps $(f: A \to B, g: M \to N)$ to $(f: A \to B, g^\circledast: N^\circledast \to M^\circledast)$ where:
\[g^{\circledast}(\phi) = \phi(g(-))\] 
So $(k\text{-}\mathsf{CALG}, \rotatebox[origin=c]{180}{$\mathbb{T}$}, \mathcal{D}, \circledast, \iota)$ is a reverse tangent category. For a commutative $k$-algebra $A$, its reverse tangent bundle is:
\[\rotatebox[origin=c]{180}{$\mathsf{T}$}^\circledast(A) = A^\circledast[\varepsilon] = \lbrace a + \phi \varepsilon \vert~ a \in A, \phi: A \to k \vert~  \text{$\phi$ $k$-linear and $\varepsilon^2 =0$} \rbrace\] 
Now consider a $k$-algebra morphism $f: A \to B$, and note that:
\[A \times_B \rotatebox[origin=c]{180}{$\mathsf{T}$}^\circledast(B) =  \lbrace a + \phi \varepsilon \vert~ a \in A, \phi: V \to k \text{ where }  \text{$\phi$ $k$-linear and $\varepsilon^2 =0$} \rbrace\] 
So $\mathsf{T}^\circledast(f): A \times_B \rotatebox[origin=c]{180}{$\mathsf{T}$}^\circledast(B)  \to \rotatebox[origin=c]{180}{$\mathsf{T}$}^\circledast(A) $ is defined as:
\[\rotatebox[origin=c]{180}{$\mathsf{T}$}^\circledast(f)(a + \phi \varepsilon) = a + \phi\left( f(-) \right) \epsilon\] 
This example generalizes nicely to the star autonomous setting. Indeed, the category of commutative monoids of any star autonomous category with distributive finite biproducts is a reverse tangent category -- thus providing a bountiful source of examples of reverse tangent categories. 
\end{example}

\begin{example} \label{ex:affine-rev} \normalfont Similarly, all of $k\text{-}\mathsf{CALG}^{op}$ will not form a reverse tangent category; instead, we must restrict to the subcategory of \emph{smooth} algebras and the display system given by the finitely generated projective modules. This requires some setup. So briefly, a smooth $k$-algebra is a commutative $k$-algebra $A$ whose associated affine scheme is smooth \cite[Def 10.137.1]{stacks-project}. Let $k\text{-}\mathsf{SmALG}$ be the full subcategory of $k\text{-}\mathsf{CALG}$ whose objects are the smooth $k$-algebras. Then to explain why $k\text{-}\mathsf{SmALG}$ is a tangent category, we need to explain why tangent algebras of smooth algebras are again smooth. Firstly, for a smooth $k$-algebra $A$ and a finitely generated projective module $M$, $\mathsf{Sym}_A(M)$ is a smooth $A$-algebra \cite[Prop 17.3.8]{grothendieck1966elements}. Moreover, \cite[Lemma 10.137.14]{stacks-project} implies that smoothness is preserved via restriction of scalars, and therefore $\mathsf{Sym}_A(M)$ is also a smooth $k$-algebra. Now for a smooth $k$-algebra $A$, $\Omega(A)$ is a finitely generated projective $A$-module \cite[Sec 10.142.(2)]{stacks-project}. Therefore, we have that $\mathsf{T}(A)$ is a smooth $k$-algebra (and a smooth $A$-algebra). Furthermore, smoothness is also preserved via change of basis \cite[ Lemma 10.137.4]{stacks-project}, which implies that $\mathsf{T}_n(A)$ is also a smooth $k$-algebra (and a smooth $A$-algebra). So we conclude that $\left(k\text{-}\mathsf{SmALG}^{op}, \mathbb{T} \right)$ is indeed a tangent category, and the differential bundles will again correspond to modules. For our system of differential bundles $\mathcal{D}$, we take the class of differential bundles that correspond to finitely generated projective modules. That $\mathcal{D}$ is indeed a system of differential bundles follows from the fact that the module of Kähler differentials of a smooth algebra is finitely generated projective, so $\mathcal{D}$ is closed under tangent bundles, and since the extension of scalars preserves finitely generated projective modules \cite[Excercise 8.4]{farb2012noncommutative}, this implies that $\mathcal{D}$ is also closed under pullback (since the pullback in $k\text{-}\mathsf{CALG}^{op}$ in this case is given by $B \otimes_A \mathsf{Sym}_A(M) \cong \mathsf{Sym}_B(B \otimes_A M)$). We can give an alternative description of $\mathsf{DBun}\left[(k\text{-}\mathsf{SmALG}, \mathbb{T}) \right]_\mathcal{D}$ similar to that of Ex \ref{ex:affine-mod}. Let $\mathsf{FMOD}$ be the full subcategory of $\mathsf{MOD}$ whose objects are the pairs $(A,M)$ consisting of a smooth $k$-algebra and a finitely generated projective $A$-module $M$\footnote{It is important to note the difference between Ex \ref{ex:cring-rev} and Ex \ref{ex:affine-rev}. In the former, we consider modules that are finitely generated over the base ring, while in the latter we take modules that are finitely generated over the algebra.}. Then we have that $\mathsf{DBun}\left[(k\text{-}\mathsf{SmALG}, \mathbb{T}) \right]_\mathcal{D} \simeq \mathsf{FMOD}^{op}$. Now the dual fibration $\left(\mathsf{FMOD}^{op}\right)^\circ$ is easier understood via its opposite category ${\left(\mathsf{FMOD}^{op}\right)^\circ}^{op}$. This category has the same objects as $\mathsf{FMOD}$, but where a map $(f,g): (A,M) \to (B,N)$ consists a $k$-algebra morphism $f: A \to B$ and an $A$-linear map $g: N \to B \otimes_A M$ in the sense that $g(f(a) n) = a g(n)$. For an $A$-module $M$, consider its $A$-linear dual, which we denote as: 
\[M^\ast = \lbrace \phi: M \to A \vert~ \text{$\phi$ $A$-linear} \rbrace\] 
Since $M$ is a finitely generated free $A$-module, it is reflexive as an $A$-module so we have the canonical isomorphism $M \cong M^{\ast\ast}$. As such, this induces an involution which sends objects $(A,M)$ to $(A, M^\ast)$. To describe what it does on a map, recall that a finitely generated projective $A$-module $M$ has a finite generating set $\lbrace e_1, ..., e_n \rbrace$ and also a finite generating set $\lbrace e^\ast_1, ..., e^\ast_n \rbrace$ for $M^\ast$. Then the involution takes a map $(f: A \to B, g: M \to N)$ to $(f: A \to B, g^\ast: N^\ast \to B \otimes_A M^\ast)$ where $g^\ast$ is defined as follows:  
\[ g^\ast(\phi) = \sum\limits^n_{i=1} \phi(g(e_i)) \otimes_A e^\ast_i \]
So $(k\text{-}\mathsf{SmALG}^{op}, \mathbb{T}, \mathcal{D}, \ast, \iota)$ is a reverse tangent category. For a commutative $k$-algebra $A$, it is well known that the $A$-linear dual of $\Omega(A)$ is $\mathsf{DER}(A)$ the module of derivations on $A$ -- which recall is a $k$-linear map $\mathsf{D}: A \to A$ which satisfies the Leibniz rule:
\[ \mathsf{D}(ab) = a \mathsf{D}(b) + b \mathsf{D}(a) \]
Therefore, for a smooth $k$-algebra $A$, we may take its reverse tangent bundle to be given as:
\[\mathsf{T}^\ast(A) = \mathsf{Sym}_A\left( \mathsf{DER}(A) \right) \] 
To describe what the reverse tangent bundle does on an algebra morphism, we must use the fact since $A$ is a smooth $k$-algebra morphism it is, by definition, a finitely presented $k$-algebra so $A \cong k[x_1, \hdots, x_n]/I$ where $k[x_1, \hdots, x_n]$ is the polynomial ring and $I$ is a finitely generated ideal $I = \left( p_1(\vec x), \hdots, p_m(\vec x) \right)$. Abusing notation slightly, let $\partial_i: A \to A$ be the derivation on $A$ associated with differentiating polynomials with respect to the $x_i$ variable (where we again abuse notation slightly and take $x_i \in A$). Then for a $k$-algebra morphism $f: A \to B$, applying $\mathsf{T}^\ast$ to it gives a $k$-algebra morphism of type $\mathsf{T}^\ast(f):  \mathsf{T}^\ast(B) \to B \otimes_A \mathsf{T}^\ast(A)$ defined as follows on generators $b \in B$ and $\mathsf{D} \in \mathsf{DER}(B)$: 
\begin{align*}
    \mathsf{T}^\ast(f)(b) = b \otimes_A 1 && \mathsf{T}^\ast(f)\left( \mathsf{D} \right) =  \sum\limits^n_{i=1} \mathsf{D}\left( f(x_i) \right) \otimes_A \partial_i 
\end{align*}
\end{example} 

\section{Some Theory of Reverse Tangent Categories}\label{sec:theory}

In this section, we provide some basic results that apply in any reverse tangent category. In particular, these results generalize key concepts about the cotangent bundle from classical differential geometry. 

We begin with the notion of a covector field.  In differential geometry, covector fields correspond to differential 1-forms. A key property of covector fields in differential geometry is that they can be pulled back; that is, a covector field on the codomain of a map can be pulled back to a covector field on the domain. The same is true in a reverse tangent category. 

\begin{definition} In a reverse tangent category $(\mathbb{X}, \mathbb{T}, \mathcal{D}, \ast, \iota)$, a \textbf{covector field} of an object $A$ is a section of $\mathsf{p}^\ast_A: \mathsf{T}^\ast(A) \to A$, that is, a map $\omega: A \to \mathsf{T}^\ast(A)$ such that the following diagram commutes: 
   \begin{equation}\label{eq:covf}\begin{gathered} \xymatrixcolsep{5pc}\xymatrix{  A \ar@{=}[dr]_-{} \ar[r]^-{\omega} & \mathsf{T}^\ast(A)  \ar[d]^-{\mathsf{p}^\ast_A}   \\
   & A  } 
   \end{gathered}\end{equation}
\end{definition}

\begin{proposition} (Pullback of covector fields) In a reverse tangent category $(\mathbb{X}, \mathbb{T}, \mathcal{D}, \ast, \iota)$, if $\omega: B \to \mathsf{T}^\ast(B)$ is a covector field, then for any map $f: A \to B$, the composite 
    \begin{equation}\label{covec-pb}\begin{gathered} \xymatrixcolsep{5pc}\xymatrix{ A \ar[r]^-{\left \langle 1_A, f \omega \right \rangle} & A \times_{B}  \mathsf{T}^\ast(B) \ar[r]^-{ \mathsf{T}^\ast(f)} &  \mathsf{T}^\ast(A) } 
    \end{gathered}\end{equation}
is a covector field on $A$.
\end{proposition}
\begin{proof} Since $(f,  \mathsf{T}^\ast(f))$ is a map in $\mathsf{DBun}^\circ_\D\left[(\mathbb{X}, \mathbb{T}) \right]$, by the left diagram of (\ref{Adiffbunmap-dual}) we easily compute that $\left \langle 1_A, f \omega \right \rangle  \mathsf{T}^\ast(f) \mathsf{p}^\ast_A = \left \langle 1_A, f \omega \right \rangle \pi_0 = 1_A$, as required. \end{proof}

Our next observation is that there is a canonical isomorphism between the composites of the tangent bundle and the reverse tangent bundle, generalizing a result for smooth manifolds \cite[Sec 26.11]{natural}.

\begin{proposition} In a reverse tangent category $(\mathbb{X}, \mathbb{T}, \mathcal{D}, \ast, \iota)$, for every object $A$, there is a natural linear $\mathsf{T}(A)$-differential bundle isomorphism $\mathsf{c}^\ast_A: \overline{\mathsf{T}}\left( \mathcal{T}^\ast(A) \right) \to \mathcal{T}^\ast(\mathsf{T}(A))$, so in particular the following diagram commutes: 
   \begin{equation}\label{c*}\begin{gathered} \xymatrixcolsep{5pc}\xymatrix{  \mathsf{T}\mathsf{T}^\ast(A) \ar[rr]^-{\mathsf{c}^\ast_A} \ar[dr]_-{\mathsf{T}(\mathsf{p}^\ast_A)} && \mathsf{T}^\ast \mathsf{T}(A)  \ar[dl]^-{\mathsf{p}^\ast_{\mathsf{T}(A)}}  \\
   & \mathsf{T}(A)  } 
   \end{gathered}\end{equation}
\end{proposition}
\begin{proof} In any tangent category, the canonical flip is an $A$-linear differential bundle isomorphism $\mathsf{c}_A: \mathcal{T}\left( \mathsf{T}(A) \right) \to  \overline{\mathsf{T}}\left( \mathcal{T}(A) \right)$. So $\mathsf{c}_A$ is an isomorphism in $\mathsf{DBun}_\mathcal{D}\left[A\right]$. Now by (\ref{T*-commute}), the dual bundle of $\mathcal{T}\left( \mathsf{T}(A) \right)$ is the tangent bundle of the reverse tangent bundle; that is, $\overline{\mathsf{T}}\left( \mathcal{T}^\ast(A) \right)$ with projection $\mathsf{T}(\mathsf{p}^\ast_A)$, while the dual bundle of $\mathcal{T}\left( \mathsf{T}(A) \right)$ is the reverse tangent bundle of the tangent bundle $\mathcal{T}^\ast(\mathsf{T}(A))$ with projection $\mathsf{p}^\ast_{\mathsf{T}(A)}$. As such, applying the induced involution $(-)^\ast: \mathsf{DBun}^{op}_\mathcal{D}\left[A\right] \to \mathsf{DBun}_\mathcal{D}\left[A\right]$ to $\mathsf{c}_A$, we obtain an $A$-linear differential bundle isomorphism $\mathsf{c}^\ast_A: \overline{\mathsf{T}}\left( \mathcal{T}^\ast(A) \right) \to \mathcal{T}^\ast(\mathsf{T}(A))$. So in particular we have an isomorphism $\mathsf{c}^\ast_A: \mathsf{T}\mathsf{T}^\ast(A) \to \mathsf{T}^\ast \mathsf{T}(A)$ such that $\mathsf{c}^\ast_A \mathsf{p}^\ast_{\mathsf{T}(A)} = \mathsf{T}(\mathsf{p}^\ast_A)$, as required. 
\end{proof}

As mentioned in Remark \ref{rem:reverse_functor}, the reverse tangent bundle does not induce a functor on the base category. However, in differential geometry, while the cotangent bundle is similarly not functorial on all smooth functions, it is functorial on \emph{\'{e}tale} maps \cite[pg. 346]{natural}, which are useful generalizations of isomorphisms. We will now show that the same is true for the reverse tangent bundle in any reverse tangent category. A smooth function $f: M \to N$ between smooth manifolds is \'{e}tale if for any $x \in M$, the tangent space at $x$ is isomorphic to the tangent space at $f(x)$, that is $\mathsf{T}_x(M) \cong \mathsf{T}_{f(x)}(N)$. However, being \'{e}tale can also be characterized in terms of a pullback square, specifically that the naturality square of the tangent bundle projection is a pullback. As such, the notion of an \'{e}tale map can be easily defined in an arbitrary tangent category. 

\begin{definition}
In a tangent category $(\X, \mathbb{T})$, a map $f: A \to B$ is \textbf{\'{e}tale} if its associated $\mathsf{p}$ naturality square: 
    \begin{equation}\label{etale}\begin{gathered} \xymatrixcolsep{5pc}\xymatrix{ \mathsf{T}(A) \ar[r]^-{\mathsf{T}(f)} \ar[d]_-{\mathsf{p}_A} & \mathsf{T}(B) \ar[d]^-{p_B} \\ A \ar[r]_-{f} & B} 
        \end{gathered}\end{equation}
is a pullback. Let $(\X, \mathbb{T})_{\text{\'{e}tale}}$ be the subcategory of \'{e}tale maps of $(\X, \mathbb{T})$. 
\end{definition}

It is straightforward to see that identity maps (and isomorphisms) are \'{e}tale, and that the composition of \'{e}tale maps is again \'{e}tale. So $(\X, \mathbb{T})_{\text{\'{e}tale}}$ is indeed well-defined.  A slightly less obvious result is the following:

\begin{lemma}\label{lem:etale-pb} In a tangent category $(\X, \mathbb{T})$, if a map $f: A \to B$ is \'{e}tale and the square
\[
\xymatrixcolsep{5pc}\xymatrix{X \ar[r]^-{\pi_1} \ar[d]_-{\pi_0} & A \ar[d]^-{f} \\ C \ar[r]_-{g} & B}
\]
is a pullback diagram which is preserved by $\mathsf{T}$, then $\pi_0: X \to C$ is also \'{e}tale.  
\end{lemma}
\begin{proof}
Consider the diagram
\[
\xymatrixcolsep{5pc}\xymatrix{\mathsf{T}(X) \ar[r]^-{\mathsf{p}_X} \ar[d]_-{\mathsf{T}(\pi_0)} & X \ar[r]^-{\pi_1} \ar[d]^-{\pi_0} & A \ar[d]^-{f} \\ \mathsf{T}(C) \ar[r]_-{\mathsf{p}_C} & C \ar[r]_{g} & B}
\]
We need to show that the left square is a pullback.  However, the right square is a pullback by assumption, so by the pullback pasting lemma, it suffices to show that the outer rectangle is a pullback.  However, by naturality of $\mathsf{p}$, the outer rectangle can be rewritten as 
\[
\xymatrixcolsep{5pc}\xymatrix{\mathsf{T}(X) \ar[r]^-{\mathsf{T}(\pi_1)} \ar[d]_-{\mathsf{T}(\pi_0)} & \mathsf{T}(A) \ar[r]^-{\mathsf{p}_A} \ar[d]^-{\mathsf{T}(f)} & A \ar[d]^-{f} \\ \mathsf{T}(C) \ar[r]_-{\mathsf{T}(g)} & \mathsf{T}(B) \ar[r]_{\mathsf{p}_B} & B
}
\]
This is a pullback since the left square is a pullback by assumption and the right square is a pullback since $f$ is \'{e}tale. Thus the left square in the first diagram is a pullback, as required.   
\end{proof}

In order to show that the reverse tangent bundle gives a functor on the subcategory of \'{e}tale maps, we will also need the following useful observation: 

\begin{lemma}\label{lem:cart-useful} Let $\D$ be a system of differential bundles on a tangent category $(\mathbb{X}, \mathbb{T})$, and suppose that $(f: A \to A^\prime, g: A \times_{A^\prime} E^\prime \to E): \mathcal{E} \to \mathcal{E}^\prime$ is a Cartesian map in $\mathsf{DBun}^\circ_\mathcal{D}\left[(\mathbb{X}, \mathbb{T}) \right]$. Then $g$ is an isomorphism, so $A \times_{A^\prime} E^\prime \cong E$, and $(f: A \to A^\prime, g^{-1}\pi_1: E \to E^\prime): \mathcal{E} \to \mathcal{E}^\prime$ is a Cartesian map in $\mathsf{DBun}_\mathcal{D}\left[(\mathbb{X}, \mathbb{T}) \right]$. 
\end{lemma} 
\begin{proof} Cartesian maps in the dual fibration correspond to Cartesian maps in the starting fibration \cite[Prop. 3.2]{dual_fibration}. In general, an equivalence class $[(v,c)]$ is Cartesian in the dual fibration if and only if the vertical component $v$ is an isomorphism. As such, we have that $[(v,c)] = [(1,v^{-1}c)]$, where in particular, $v^{-1}c$ is a Cartesian map in the starting fibration. Thus, the desired result is obtained by translating this in terms of $\mathsf{DBun}_\mathcal{D}\left[(\mathbb{X}, \mathbb{T}) \right]$. So if $(f,g)$ is Cartesian in $\mathsf{DBun}^\circ_\mathcal{D}\left[(\mathbb{X}, \mathbb{T}) \right]$, then its associated representation as an equivalence class is $[ (1_A,g), (f,\pi_1) ]$. This implies that $(1_A, g)$ is an isomorphism in $\mathsf{DBun}_\mathcal{D}\left[(\mathbb{X}, \mathbb{T}) \right]$, which in turn implies that $g$ is an isomorphism. Then $[ (1_A,g), (f,\pi_1) ] = [ (1_A, 1_E), (f, g^{-1} \pi_1) ]$, and in particular $(f,g^{-1} \pi_1)$ is a Cartesian map in $\mathsf{DBun}_\mathcal{D}\left[(\mathbb{X}, \mathbb{T}) \right]$. 
\end{proof}

We can now prove our main result about \'{e}tale maps:

\begin{proposition}\label{prop:etale} (Functoriality of $\mathsf{T}^\ast$ on \'{e}tale maps) In a reverse tangent category $(\mathbb{X}, \mathbb{T}, \mathcal{D}, \ast, \iota)$, if $f: A \to B$ is \'{e}tale, then $\mathsf{T}^\ast(f): A \times_B \mathsf{T}^\ast(B) \to \mathsf{T}^\ast(A)$ is an isomorphism. Furthermore, there is an endofunctor $\widehat{\mathsf{T}^\ast}: (\X, \mathbb{T})_{\text{\'{e}tale}} \to (\X, \mathbb{T})_{\text{\'{e}tale}}$ which maps an object $A$ to its reverse tangent bundle $\mathsf{T}^\ast(A)$, and an \'{e}tale map $f: A \to B$ to the map $\widehat{\mathsf{T}^\ast}(f): \mathsf{T}^\ast(A) \to \mathsf{T}^\ast(B)$ which is defined as the composite $\widehat{\mathsf{T}^\ast}(f) := \mathsf{T}^\ast(f)^{-1} \pi_1$. 
\end{proposition}\label{etale_functor}
\begin{proof} Since $f$ is \'{e}tale, $(f, \mathsf{T}(f))$ is a pullback square, and hence Cartesian in $\mathsf{DBun}_\mathcal{D}\left[(\mathbb{X}, \mathbb{T}) \right]$. Since the involution $(-)^\ast$ is a fibration morphism, it sends Cartesian maps to Cartesian maps. As such, $(f, \mathsf{T}(f))^\ast= (f, \mathsf{T}^\ast(f))$ is Cartesian in $\mathsf{DBun}^\circ_\mathcal{D}\left[(\mathbb{X}, \mathbb{T}) \right]$. By Lemma \ref{lem:cart-useful}, it follows that $\mathsf{T}^\ast(f)$ is an isomorphism.  Thus $\mathsf{T}^\ast(f)^{-1}$ is also an isomorphism and hence \'{e}tale, and by Lemma \ref{lem:etale-pb}, $\pi_1: A \times_B \mathsf{T}^\ast(B) \to \mathsf{T}^\ast(B)$ is also \'{e}tale since it is a pullback of the \'{e}tale map $f$. Hence $\widehat{\mathsf{T}^\ast}(f) := \mathsf{T}^\ast(f)^{-1} \pi_1$ is itself \'{e}tale, as it is a composite of \'{e}tale maps.  It is then easy to check this assignment is functorial.    
\end{proof}

In Ex \ref{ex:CRDC}, we explained how every Cartesian reverse differential category is a reverse tangent category. We conclude this section by going in the opposite direction. Looking at the forward side of the story, to do so we must work with Cartesian tangent categories \cite[Def 2.8]{cockett2014differential}, which are tangent categories with finite products that are preserved by the tangent bundle functor $\mathsf{T}$. Then to extract a Cartesian differential category from a Cartesian tangent category, one looks at the \emph{differential objects} \cite[Sec 3]{cockett2018differential}, which are the differential bundles over the terminal object $\mathsf{1}$. In particular, a differential object $A$ has the property that $\mathsf{T}(A) \cong A \times A$. Then the full subcategory $\mathsf{DO}$ of differential objects and \emph{all} maps between them is a Cartesian differential category \cite[Thm 4.11]{cockett2014differential}, where the differential combinator is defined on a map $f: A \to B$ as follows: 
    \begin{equation}\label{diff-do}\begin{gathered} \mathsf{D}[f] := \xymatrixcolsep{5pc}\xymatrix{ A \times A \cong \mathsf{T}(A) \ar[r]^-{ \mathsf{T}(f)} &  \mathsf{T}(B) \cong B \times B \ar[r]^-{\pi_1} & B} 
    \end{gathered}\end{equation}
As differential objects give a Cartesian differential category, to obtain a Cartesian reverse differential category, we need only give a linear dagger, which will be built using the linear involution. First note that for a Cartesian tangent category, for every object $A$, $\mathsf{DBun}[A]$ has finite biproducts. Now one of the axioms of a linear dagger is that the fibres of the linear fibration have dagger biproducts. Thus, we must ask that the involution preserves this biproduct structure: 

\begin{definition} A \textbf{Cartesian reverse tangent category} is a reverse tangent category $(\mathbb{X}, \mathbb{T}, \mathcal{D}, \ast, \iota)$ such that $(\mathbb{X}, \mathbb{T})$ is a Cartesian tangent category, $\mathcal{D}$ is closed under products, and for each object $A$, the induced involution $(-)^\ast: \mathsf{DBun}^{op}_\mathcal{D}\left[A\right] \to \mathsf{DBun}_\mathcal{D}\left[A\right]$ preserves the biproduct structure. 
\end{definition}

In a Cartesian reverse differential category, since the dagger is the identity on objects, we have that the differential bundles in our chosen system are in fact self-dual. So to build a Cartesian reverse differential category from a Cartesian reverse tangent category, we consider the differential objects which are also self-dual. 

\begin{definition} In a Cartesian reverse tangent category $(\mathbb{X}, \mathbb{T}, \mathcal{D}, \ast, \iota)$, a \textbf{self-dual differential object} is a differential object $A$ in $\mathcal{D}$ which comes equipped with a linear isomorphism $A \cong A^\ast$. Let $\mathsf{DO}_{\text{sd}}$ be the full subcategory of self-dual differential objects and all maps between them. 
\end{definition}

\begin{lemma} In a Cartesian reverse tangent category $(\mathbb{X}, \mathbb{T}, \mathcal{D}, \ast, \iota)$, if $A$ is a $\ast$-self-dual differential object, then $\mathsf{T}^\ast(A) \cong A \times A \cong \mathsf{T}(A)$. 
\end{lemma}
\begin{proof} Since $A$ is a differential object and thus a differential bundle, by applying (\ref{T*-commute}), we get that $\mathsf{T}^\ast(A) \cong \mathsf{T}(A^\ast) \cong \mathsf{T}(A) \cong A \times A$. 
\end{proof}

We can now explain how the full subcategory of self-dual differential objects is a Cartesian reverse differential category. 

\begin{proposition}\label{prop:CRTC-CRDC} For a Cartesian reverse tangent category $(\mathbb{X}, \mathbb{T}, \mathcal{D}, \ast, \iota)$, $\mathsf{DO}_{\text{sd}}$ is a Cartesian reverse differential category where the reverse differential combinator $\mathsf{R}$ is defined on a map $f: A \to B$ as the following composite: 
    \begin{equation}\label{rev-do}\begin{gathered} \mathsf{R}[f] := \xymatrixcolsep{5pc}\xymatrix{ A \times B \cong A \times_{B}  \mathsf{T}^\ast(B) \ar[r]^-{ \mathsf{T}^\ast(f)} &  \mathsf{T}^\ast(A) \cong A \times A \ar[r]^-{\pi_1} & A } 
    \end{gathered}\end{equation}
\end{proposition}
\begin{proof} Since $\mathsf{DO}_{\text{sd}}$ is a full subcategory of $\mathsf{DO}$, it is clear that $\mathsf{DO}_{\text{sd}}$ is also a Cartesian differential category with the same structure as $\mathsf{DO}$. So it remains to construct a linear dagger for $\mathsf{DO}_{\text{sd}}$. To do so, given a map $f: C \times A \to B$ which is linear in $A$, we must give a map $f^\dagger: C \times B \to A$ which is linear in $B$. Now $C \times A$ and $C \times B$ are differential bundles over $C$ and $\langle \pi_0, f \rangle: C \times A \to C \times B$ is a map in $\mathsf{DBun}_\mathcal{D}\left[C\right]$. Applying the involution $(-)^\ast: \mathsf{DBun}^{op}_\mathcal{D}\left[A\right] \to \mathsf{DBun}_\mathcal{D}\left[A\right]$, we obtain a map $(\langle \pi_0, f \rangle)^\ast: C \times B^\ast \to C \times A^\ast$. Then define $f^\dagger$ as the following composite: 
  \begin{equation}\label{fdag-def}\begin{gathered} f^\dagger := \xymatrixcolsep{5pc}\xymatrix{ C \times B \cong C \times B^\ast \ar[r]^-{ (\langle \pi_0, f \rangle)^\ast} &  C \times A^\ast \cong C \times A \ar[r]^-{\pi_1} & A } 
    \end{gathered}\end{equation}
This map is linear in $B$ since $(\langle \pi_0, f \rangle)^\ast$ is a $C$-linear differential bundle morphism. It is straightforward to check that this induces a linear dagger $\dagger$ on $\mathsf{DO}_{\text{sd}}$. So we conclude that $\mathsf{DO}_{\text{sd}}$ is a Cartesian reverse differential combinator, and the reverse differential combinator defined in (\ref{rev-do}) is precisely the dagger of the differential combinator defined in (\ref{diff-do}).
\end{proof}

\begin{example} $(\mathsf{SMAN}, \mathbb{T}, \mathcal{D}, \ast, \iota)$ is a Cartesian reverse tangent category, and its $\ast$-self-dual differential objects are precisely the Euclidean spaces $\mathbb{R}^n$. As such, the resulting Cartesian reverse differential category is precisely $\mathsf{SMOOTH}$. 
\end{example}

\section{Future Work}\label{sec:future_work}

One of the next major steps for reverse tangent categories is to apply these ideas to categorically study gradient-based learning and automatic differentiation on smooth manifolds. However, there are several other ways this work could be expanded upon: 
\begin{enumerate}[{\em (i)}]
    \item We have chosen to \emph{define} a reverse tangent category as a tangent category with a certain kind of involution.  However, this was not how Cartesian reverse differential categories were defined.  Cartesian reverse differential categories were defined directly in terms of a reverse differential combinator $\mathsf{R}$, and then shown to be equivalent to a Cartesian differential category with a certain kind of involution. It would be useful to have a direct description of a reverse tangent category in a similar fashion, that is, a structure involving a ``reverse tangent bundle'' functor from the base category to an appropriately defined category of differential bundles.  
    \item There is much more theoretical work that can be explored in an arbitrary reverse tangent category, especially by taking inspiration from results about the cotangent bundle in differential geometry.  For example, a pseudo-Riemannian structure on a manifold can be defined as an isomorphism between its tangent bundle and its cotangent bundle; thus, one could similarly explore what can be done with such objects in an arbitrary reverse tangent category.
    \item There are many ways in which tangent categories can be generated from existing ones. For example, the category of vector fields of a tangent category is again a tangent category \cite[Prop 2.10]{diff_eqns}.  It would be interesting to see how many of these constructions apply to reverse tangent categories, thus giving many more examples of this structure.  
\end{enumerate}

\bibliography{reverse_bib}  
\end{document}